\documentclass[final,3p]{elsarticle}
\usepackage{amssymb}
\usepackage{amsmath}
\usepackage{amsthm}
\usepackage[T1]{fontenc}
\usepackage{graphicx}
\usepackage{color}
\usepackage{amsthm}
\usepackage{graphics}
\usepackage{amsmath,amssymb,amsfonts}
\usepackage{epsfig}

\newcommand{\C}{\mathbb{C}}
\newcommand{\Z}{\mathbb{Z}}

\newtheorem{Definition}{Definition}
\newtheorem{theorem}{Theorem}
\newtheorem{conjecture}[theorem]{Conjecture}
\newtheorem{lemma}{Lemma}
\newtheorem{remark}{Remark}
\newtheorem{proposition}{Proposition}

\newcommand{\diagi}{\begin{smallmatrix}\vspace{-0.5ex}\textrm{\normalsize diag}\\\vspace{-0.8ex}i\in\mathcal{I}_{n}\end{smallmatrix}}
\newcommand{\diagik}{\begin{smallmatrix}\vspace{-0.5ex}\textrm{\normalsize diag}\\\vspace{-0.8ex}i\in\mathcal{I}_{k}\end{smallmatrix}}
\newcommand{\diff}[2]{\frac{\partial {#1} }{\partial {#2} } }
\newcommand{\bx}[1]{{\bf #1}} 

\begin{document}

\title{Multigrid methods for block-Toeplitz linear systems:\\ convergence analysis and applications}
\author{Marco Donatelli, Paola Ferrari, Isabella Furci, Stefano Serra Capizzano, Debora Sesana}

\begin{abstract}
In the past decades, multigrid methods for linear systems having multilevel Toeplitz coefficient matrices with scalar entries have been largely studied. On the other hand, only few papers have investigated the case of block entries, where the entries are small generic matrices instead of scalars. In that case the efforts of the researchers have been mainly devoted to specific applications, focusing on algorithmic proposals but with very marginal theoretical results.

In this paper, we propose a general two-grid convergence analysis proving an optimal convergence rate independent of the matrix size, in the case of positive definite block Toeplitz matrices with generic blocks. In particular, the proof of the approximation property has not a straightforward generalization of the scalar case and in fact we have to require a specific commutativity condition on the block symbol of the grid transfer operator. 
Furthermore, we define a class of grid transfer operators satisfying the previous theoretical conditions and we propose a strategy to insure fast multigrid convergence even for more than two grids.

Among the numerous applications that lead to the block Toeplitz structure, high order Lagrangian finite element methods and staggered discontinuous Galerkin methods are considered in the numerical results, confirming the effectiveness of our proposal and the correctness of the proposed theoretical analysis.
\end{abstract}

\maketitle

\section{Introduction}
\label{sec:circ}

We are interested in solving large positive definite linear systems arising from particular finite element approximation of partial differential equations (PDEs) or coupled systems. As examples we consider the quadrilateral Lagrangian finite element methods (FEM) and staggered discontinuous Galerkin (DG) methods  for the incompressible Navier-Stokes equations, see \cite{qp, DMPS, DFFMST}. In these applications, when the PDE has constant coefficients, the resulting matrices possess a natural block-Toeplitz structure, up to a low rank correction due to boundary conditions. A block-Toeplitz matrix has a Toeplitz structure (constant entries along the diagonals), where the entries are generic small $d \times d$ matrices instead of scalars. With reference to the block-Toeplitz character, other possible applications are coupled systems of integro-differential equations: whenever the discretization of each equation has a Toeplitz structure, rearranging the unknowns by a proper permutation, the associated linear system retrieves the already mentioned block-Toeplitz structure, see \cite{D3S,MRS}.

After the seminal papers on multigrid methods for Toeplitz matrices  investigated in
\cite{FS1,CCS}, the results have been extended to multidimensional problems including the V-cycle convergence analysis, see \cite{ADS} and reference therein. A multigrid methods for block Toeplitz matrices has been proposed in \cite{HS2} and studied in the case of diagonal block symbol (defined below). This was then adapted and further analysed for specific applications, like those considered in \cite{DMPS, DDMNS}, but the results are strictly related to the block (multilevel) Toeplitz matrices in question. In practice, when the block symbol is not diagonal, there is still a substantial  lack of an effective projection proposal and of a rigorous convergence analysis. 

The first aim of the paper is to generalize the existing convergence results in the scalar settings for systems with coefficient matrix in the circulant algebra associated with a matrix-valued symbol. According to the relevant literature, the classical Ruge and St\"uben convergence analysis in \cite{RStub} is applied in order to split the two-grid convergence in smoothing property and approximation property. The smoothing property is proved for damped Jacobi with the relaxation parameter chosen in an interval depending on the symbol. The proof of the approximation property provides a generalization of the two conditions present in the scalar and requires a further commutativity condition on the matrix-valued symbol of the grid transfer operator. In order to extend the results to V-cycle, we propose a measure of the ill-conditioning of the symbol at the coarser levels in order to choose a robust grid transfer operator.

We exploit the algebra structure of circulant matrices for the theoretical analysis of the two-grid and V-cycle algorithms, and we consider Toeplitz matrices for practical applications. This is a common approach and it is supported by the fact that the symbol analysis for Toeplitz matrices is an algebraic generalization of the local Fourier analysis of multigrid methods, see \cite{D}.

Finally, we present some numerical results for quadrilateral Lagrangian FEM and staggered discontinuous Galerkin methods for the incompressible Navier-Stokes equations. The results confirm the theoretical analysis proving an optimal convergence rate also for the V-cycle. Here for optimal rate we mean that the convergence speed is linear and independent of the matrix size and often mildly depending on other relevant parameters such as the dimensionality of the domain or the polynomial degree in the considered FEM/DG methods. 

The paper is organized as follows.
In Section \ref{sect:notation} we fix the notation and we recall the main properties of block-circulant matrices with their main algebraic, structural, and spectral properties.
In Section \ref{sect:two_grid} we give an overview on the two-grid method with a particular focus on the convergence results.
In Section \ref{sect:project_circ} we  briefly sketch the basic ideas for defining the projecting operators for block-circulant matrices. 
Convergence analysis and optimality proof of the two-grid technique are reported in Section \ref{sect:proof_circ}. As a conclusion of the theoretical analysis, we define an ill-conditioning of the coarse problem in order to choose a robust grid transfer operator for the V-cycle method.
In Section \ref{sect:experiments} we study the applicability of our two-grid and V-cycle procedures to linear systems stemming from the approximation of differential operators. In particular, in Subsections \ref{subsection:FEM}-\ref{subsection:FEM2D} we will report numerical results for the $\mathbb{Q}_{\deg}$ Lagrangian FEM applied to the Poisson problem. In Subsection \ref{subsection:staggeredDG} we will focus instead on the matrices arising from the discretization by staggered discontinuous Galerkin methods of the incompressible Navier-Stokes equations.
Section \ref{sect:final} contains conclusions and discusses few issues to be considered in future works.

\section{Notation}\label{sect:notation}

In the current section we fix the notation for matrix and function norms, matrix-valued trigonometric polynomials, block-Toeplitz and block-circulant matrices.

\subsection{Norms} 
Given $1\le p<\infty$ and a vector $x \in \mathbb{C}^{n}$, we denote by $\|x\|_{p}$ the $p$-norm of $x$ and by $|\cdot|_{p}$ the associated induced matrix norm over $\mathbb{C}^{n\times n}$. If $X$ is positive definite,
$\|\cdot\|_{X}=|X^{1/2}\cdot|_{2}$ denotes the Euclidean norm
weighted by $X$ on $\mathbb{C}^{n}$. Moreover, if we denote by $\sigma_{j}(X)$, $j=1,\ldots,n$, the singular values of a matrix $X\in\mathbb{C}^{n\times n}$, $\|\cdot\|_{1}$ is the so called trace-norm on $\mathbb{C}^{n\times n}$ defined by
$\|\cdot\|_{1}=\sum_{j=1}^{n}\sigma_{j}(\cdot)$, and 
$\|\cdot\|_{\infty}=\max_{j=1,\ldots,n}\sigma_{j}(\cdot)$ is the spectral norm. 
Finally, if $X$ and $Y$ are Hermitian matrices, then
the notation $X\leq Y$ means that $Y-X$ is nonnegative definite.

\subsection{Block-Toeplitz matrices} 
 Let $\mathcal{M}_d$
be the linear space of the complex $d\times d$ matrices and let
$f:Q\to\mathcal M_d$, with $Q=(-\pi,\pi)$. We say that $f \in L^p(d)$ (resp. is
measurable) if all its components $f_{ij}:Q\to\mathbb C,\
i,j=1,\ldots,d,$ belong to $L^p(d)$ (resp. are measurable) for $1\le
p\le\infty$. 

\begin{Definition}\label{def-multilevel}
Let the Fourier coefficients of  a given function $f$, defined as $f\in L^1(d)$, be
\begin{align*}
  \hat f_j:=\frac1{2\pi}\int_{Q}f(\theta){\rm e}^{- \iota j \theta} d\theta\in\mathcal M_d,
  \qquad  \iota^2=-1, \, j\in\mathbb Z.
\end{align*}
Then, the block-Toeplitz matrix associated with $f$ is the matrix of order $dn$ given by
\begin{align*}
  T_n(f)=\sum_{|j|<n}J_{n}^{(j)}\otimes\hat f_j,
\end{align*}
where $\otimes$ denotes the (Kronecker) tensor product of matrices. The term
$J_n^{(j)}$ is the matrix of order $n$ whose $(i,k)$ entry equals $1$ if $i-k=j$
and zero otherwise.  
\end{Definition}
The set $\{T_n(f)\}_{n\in\mathbb N}$ is
called the \textit{family of block-Toeplitz matrices generated by $f$}, that
in turn is referred to as the \textit{generating function or the symbol of
$\{T_n(f)\}_{n\in\mathbb N}$}.

\subsection{Block-circulant matrices} 
\label{subsect:background_circ_twogrid}
In the scalar case, when $d=1$, if $f$ is a polynomial we can define the circulant matrix  generated by $f$ by
\begin{align*}
  \mathcal{A}_{n}(f)=F_n \diagi(f(\theta_i^{(n)}))F_{n}^{H},
\end{align*} 
where $F_{n}=\frac{1}{\sqrt{n}}\left[e^{-\imath j\theta_i^{(n)}}\right]_{i,j=0}^{n-1}$, the grid points $\theta_i^{(n)}$ are $\frac{2\pi i}{n}$ and $i$ belongs to the index range $\mathcal{I}_{n} = \{0,\ldots,n-~1\}$. Circulant matrices form an algebra $\mathcal{C}$ of normal matrices.

In the block-case, $d>1$, if $f\in\mathcal{M}_d$ is a matrix-valued trigonometric polynomial the block-circulant matrix generated by $f$ is defined as
\begin{equation*}
  \mathcal{A}_{n}(f)=(F_n\otimes I_{d}) \diagi(f(\theta_i^{(n)}))(F_{n}^{H}\otimes I_{d}),
\end{equation*}   
where $\otimes$ is the tensor (Kronecker) product of matrices and $\diagi(f(\theta_i^{(n)}))$ is the block-diagonal matrix where the block-diagonal elements are the evaluation of $f$ on the grid points $\theta_i^{(n)}$, $i\in\mathcal{I}_{n}$. The matrix $\mathcal{A}_{n}$ has size $dn \times dn$.

\section{Two-grid method}\label{sect:two_grid}

Let $A_{n}\in\mathbb{C}^{n\times n}$, and
$x_{n},\,b_{n}\in\mathbb{C}^{n}$. Let
$p_n^k\in\mathbb{C}^{n\times k}$, $k<n$, be a given full-rank
matrix and let us consider a class of iterative methods of the form
\begin{align}\label{prepost}
  x_{n}^{(j+1)}=V_{n}x_{n}^{(j)}+\tilde{b}_{n}:=\mathcal{V}(x_{n}^{(j)},\tilde{b}_{n}),
\end{align}
where $A_{n}=W_{n}-N_{n}$, $W_{n}$ nonsingular matrix, $V_{n}:=I_{n}-W_{n}^{-1}A_{n}\in\mathbb{C}^{n\times n}$,
and $\tilde{b}_{n}:=W_{n}^{-1}b_{n}\in\mathbb{C}^{n}$.
A Two-Grid Method (TGM) is defined by the following algorithm:
\begin{center}
\begin{tabular}{l}
\vspace{1ex}
TGM$(V_{n,\rm{pre}}^{\nu_{\rm{pre}}},V_{n,\rm{post}}^{\nu_{\rm{post}}},p_{n}^{k})(x_{n}^{(j)})$
\vspace{1ex}\\
\hline
  0. $\tilde{x}_{n}=\mathcal{V}_{n,\rm{pre}}^{\nu_{\rm{pre}}}(x_{n}^{(j)},\tilde{b}_{n,\rm{pre}})$\\
  1. $d_{n}=A_{n}\tilde{x}_{n}-b_{n}$\\
  2. $d_{k}=(p_{n}^{k})^{H}d_{n}$\\
  3. $A_{k}=(p_{n}^{k})^{H}A_{n}p_{n}^{k}$\\
  4. Solve $A_{k}y=d_{k}$\\
  5. $\hat{x}_{n}=\tilde{x}_{n}-p_{n}^{k}y$\\
  6. $x_{n}^{(j+1)}=\mathcal{V}_{n,\rm{post}}^{\nu_{\rm{post}}}(\hat{x}_{n},\tilde{b}_{n,\rm{post}})$
\end{tabular}
\end{center}

Steps $1.\rightarrow5.$ define the ``coarse grid correction'' that
depends on the projecting operator $p_{n}^{k}$, while Step $0.$ and
Step $6.$ consist, respectively, in applying $\nu_{\rm{pre}}$ times
and $\nu_{\rm{post}}$ times a ``pre-smoothing iteration'' and a
``post-smoothing iteration'' of the generic form given in
$(\ref{prepost})$.
The global iteration matrix of the TGM is then given by
\begin{align*}
  {\rm TGM}(V_{n,\rm{pre}}^{\nu_{\rm{pre}}},V_{n,\rm{post}}^{\nu_{\rm{post}}},p_{n}^{k})=
  V_{n,\rm{post}}^{\nu_{\rm{post}}}
  \left[I_{n}-p_{n}^{k}\left((p_{n}^{k})^{H}
  A_{n}p_{n}^{k}\right)^{-1}(p_{n}^{k})^{H}A_{n}\right]V_{n,\rm{pre}}^{\nu_{\rm{pre}}}.
\end{align*}

In the present paper, we are interested in proposing such a kind of
techniques in the case where $A_{n}$ is a block-circulant matrix.
First we recall some general convergence results from the theory of the algebraic multigrid method given in \cite{RStub}.
For the optimality proof of TGM we need the following result, see \cite[Theorem 5.2]{RStub} and \cite[Remark 2.2]{ADS}.

\begin{theorem}\label{teoconv}
Let $A_{n}$ be a positive definite matrix of size $n$ and let $V_{n}$ be defined as in the {\rm TGM} algorithm.
Assume
\begin{itemize}
\item[(a)] $\exists\alpha_{\rm{post}}>0\,:\;\|V_{n,\rm{post}}x_{n}\|_{A_{n}}^{2}\leq\|x_{n}\|_{A_{n}}^{2}-\alpha_{\rm{post}}\|x_{n}\|_{A_{n}^2}^{2},\qquad \forall x_{n}\in\mathbb{C}^{n},$
\item[(b)] $\exists\gamma>0\,:\;\min_{y\in\mathbb{C}^{k}}\|x_{n}-p_{n}^{k}y\|_{2}^{2}\leq \gamma\|x_{n}\|_{A_{n}}^{2},\qquad \forall x_{n}\in\mathbb{C}^{n}.$
\end{itemize}
Then $\gamma\geq\alpha_{\rm{post}}$ and
\begin{align*}
  \|{\rm TGM}(I,V_{n,\rm{post}}^{\nu_{\rm{post}}},p_{n}^{k})\|_{A_{n}}\leq\sqrt{1-\alpha_{\rm{post}}/\gamma}.
\end{align*}
\end{theorem}
Conditions $(a)$ and $(b)$ are usually called ``smoothing property'' and
``approximation property'', respectively.

Since $\alpha_{\rm{post}}$ and $\gamma$ are independent of $n$, if the assumptions of Theorem \ref{teoconv} are satisfied, then the resulting TGM is not only convergent but also optimal. In other words, the number of iterations in order to reach a given accuracy $\epsilon$ can be bounded from above by a constant independent of $n$ (possibly depending on the parameter $\epsilon$).

Of course, if the given method is complemented with a convergent pre-smoother, then by the same theorem we get a faster convergence. In fact, it is known that for square matrices $A$ and $B$ the spectra of $AB$ and $BA$
coincide.

Therefore ${\rm TGM}(V_{n,\rm{pre}}^{\nu_{\rm{pre}}},V_{n,\rm{post}}^{\nu_{\rm{post}}},p_{n}^{k})$
and ${\rm TGM}(I,V_{n,\rm{pre}}^{\nu_{\rm{pre}}}V_{n,\rm{post}}^{\nu_{\rm{post}}},p_{n}^{k})$ have the same ei\-gen\-values
so that
\begin{align*}
\|{\rm TGM}(V_{n,\rm{pre}}^{\nu_{\rm{pre}}},V_{n,\rm{post}}^{\nu_{\rm{post}}},p_{n}^{k})\|_{A_{n}}
& = 
\|{\rm TGM}(I,V_{n,\rm{pre}}^{\nu_{\rm{pre}}}V_{n,\rm{post}}^{\nu_{\rm{post}}},p_{n}^{k})\|_{A_{n}}\\
& \le \sqrt{1-\alpha_{\rm{post}}^{\rm{new}}/\gamma} \le
 \sqrt{1-\alpha_{\rm{post}}/\gamma},
\end{align*}
and hence the presence of a pre-smoother can only improve the convergence.


\section{Projecting operators for block-circulant matrices}\label{sect:project_circ}

The choice the prolongation and restriction operators in order to validate the approximation condition is crucial for TGM convergence and optimality. In the current section, we define the structure of projecting operators $p_n^k$ for the block-circulant matrix $\mathcal{A}_{n}(f)$ generated by a trigonometric polynomial $f:Q\rightarrow\mathcal{M}_d$.

On the one hand, $p_n^k$ projects the problem into a coarser one, ``cutting'' the matrix $\mathcal{A}_{n}(f)$, 
on the other hand the ``cut'' and projected matrix should maintain the same structure and the properties of $\mathcal{A}_{n}(f)$. Hence, as projector $p_n^k$ we choose the product between a matrix $\mathcal{A}_{n}(p)$ in the algebra, where $p$ is a trigonometric polynomial, and a cutting matrix $K_{n}^{T}\otimes I_d$ ($K_{n}$ defined in Table \ref{CT-projector}).

The equality in Table \ref{CT-projector} line 5 plays a basic role in maintaining the matrix algebra structure on subgrids, as we will see in Proposition \ref{fhat}.
\begin{table}
\centering
\begin{tabular}{cc}
    \hline
    	Object & Definition in the circulant algebra \\
		\hline
		$n$ & $2^t$ \\
		$k$ & $\frac{n}{2}=2^{t-1}$ \\
		$K_{n}$ & $\left[\begin{array}{cccccccc}
		1 & 0 & & & & &\\
		  &   & 1 & 0 & & & & \\
			&   &   &   & \ddots & \ddots & & \\
			&   &   &   &        &        & 1 & 0		
		\end{array}\right]_{k\times n}$ \\
		$I_{n,2}$ & $\left[I_{k} | I_{k}\right]_{k\times n}$ \\
		$K_{n}F_{n}$ & {$\frac{1}{\sqrt{2}}F_{k}I_{n,2}$}\\
		$p_n^k$ & {$\mathcal{A}_{n}(p)(K_{n}^T\otimes I_d)$}\\		
		\hline
\end{tabular}
\caption{Dimensions, cutting operators and relations in the case $d=1$, $N=dn$ and $K=dk$.}\label{CT-projector}
\end{table}
We note that, from the definition of $k$ in Table \ref{CT-projector}, $n$ must be even. We are left to determine the conditions to be satisfied
by $\mathcal{A}_{n}(p)$ (or better by its generating function $p$), in order to get a projector which is effective in terms of convergence.

\subsection{TGM conditions} 
Let $A_{N}=\mathcal{A}_{n}(f)$, $N = N(d,n) = dn$, with $f$ matrix-valued trigonometric polynomial, $f\geq0$, and let $p_n^k=\mathcal{A}_{n}(p)(K_{n}^T\otimes I_d)$ with $p$ matrix-valued trigonometric polynomial. Define $\Theta_0$ as the set of points $\theta$ such that $\lambda_j(f(\theta))=0$ for some $j$. Assume that, for $\theta\in\Theta_0$, $\lambda_j(f(\theta+\pi))\not=0$ for all $j=1,\dots,d$, which also implies that the set $\Theta_0$ is a finite set. Choose $p(\cdot)$ diagonalizable such that the following relations
\begin{align}
\exists \delta \mbox{ s.t. } |f(\theta)^{-\frac{1}{2}}p(\theta+\pi)^{H}|_1<\delta \quad &\forall \theta \in [0,2\pi)\backslash \Theta_0 \label{p2f1},\\
p(\theta)^{H}p(\theta)+p(\theta+\pi)^{H}p(\theta+\pi)>0 \quad &\forall \theta\in[0,2\pi), \label{p2f3}\\
p(\theta)p(\theta+\pi) = p(\theta+\pi)p(\theta) \quad &\forall \theta\in[0,2\pi) \label{eqn:commutativity}
\end{align}
are fulfilled.\\

\begin{remark} \label{remark:commutativity}
	Notice that condition (\ref{eqn:commutativity}) implies that there exists a unitary transform $U(\cdot)$ and a diagonal matrix-valued function $D_p(\cdot)$ such that $p(\theta) = U(\theta)D_p(\theta)U(\theta)^H$ and $p(\theta+\pi) = U(\theta)D_p(\theta+\pi)U(\theta)^H$. In particular, we have
	\begin{equation*}
		(p(\theta)^{H}p(\theta)+p(\theta+\pi)^{H}p(\theta+\pi))^{-1} = U(\theta)(|D_p(\theta)|^2+|D_p(\theta+\pi)|^2)^{-1}U(\theta)^H,
	\end{equation*}
	which ensures that $(p(\theta)^{H}p(\theta)+p(\theta+\pi)^{H}p(\theta+\pi))^{-1}$ commutes with $p(\theta)$, $p(\theta)^H$, $p(\theta+\pi)$ and $p(\theta+\pi)^H$.
\end{remark}

Before proving that the above conditions are sufficient to assure the TGM optimality, we consider
a crucial result both from a theoretical and a practical point of view.

\begin{proposition}\label{fhat} Let $f$ be a nonnegative definite matrix-valued function, $k$ defined as in Table \ref{CT-projector}, $K = dk$, $p_n^k=\mathcal{A}_{n}(p)(K_{n}^T\otimes I_d)\in\mathbb{C}^{N\times K}$, with $p$ trigonometric polynomial satisfying condition $(\ref{p2f1})$ for any
zero eigenvalue of $f$ and globally the condition $(\ref{p2f3})$. Then the matrix
$(p_n^k)^H\mathcal{A}_{n}(f)p_n^k\in\mathbb{C}^{K\times K}$
coincides with $\mathcal{A}_{k}(\hat{f})$ where $\hat{f}$ is nonnegative definite and
\begin{align}\label{f2t}
  \hat{f}(\theta)=\frac{1}{2}\left(p\left(\frac{\theta}{2}\right)^{H}f\left(\frac{\theta}{2}\right)p\left(\frac{\theta}{2}\right)+
	p\left(\frac{\theta}{2}+\pi\right)^{H}f\left(\frac{\theta}{2}+\pi\right)p\left(\frac{\theta}{2}+\pi\right)\right).
\end{align}
\end{proposition}

\begin{proof} Using the notation in Table \ref{CT-projector}, we have that
\begin{align*}
  (p_n^k)^H\mathcal{A}_{n}(f)p_n^k&=(K_{n}\otimes I_d)\mathcal{A}_{n}(p^H)\mathcal{A}_{n}(f)\mathcal{A}_{n}(p)(K_{n}^T\otimes I_d)\\
	&=(K_{n}\otimes I_d)\mathcal{A}_{n}(p^Hfp)(K_{n}^T\otimes I_d)\\
	&=(K_{n}\otimes I_d)(F_{n}\otimes I_d)\diagi(p^Hfp(\theta_i^{(n)})(F_{n}^H\otimes I_d)(K_{n}^T\otimes I_d)\\
	&=\frac{1}{2}(F_{k}\otimes I_d)\diagik(p^Hfp(\theta_i^{(n)})+p^Hfp(\theta_{\tilde{i}}^{(n)}))(F_{k}^{H}\otimes I_d)\\
	&=\frac{1}{2}(F_{k}\otimes I_d)\diagik \left(p^{H}fp\left(\frac{\theta_{i}^{(k)}}{2}\right)+p^{H}fp\left(\frac{\theta_{i}^{(k)}}{2}+\pi\right)\right)(F_{k}^{H}\otimes I_{d})\\
&=\mathcal{A}_{k}(\hat{f}),
\end{align*}
where $\tilde{i}=i+k$; this is again a block-circulant matrix of size $K$. From the structure of $\hat{f}$ is clear that if $f$ is nonnegative definite also $\hat{f}$ is nonnegative definite.
\end{proof}

\section{Proof of convergence}\label{sect:proof_circ}
The current section is divided into two parts. In Subsection \ref{subsect:TGMconv} we prove the optimality of the two-grid method validating both the smoothing and the approximation conditions. In Subsection \ref{subsect:MGM} we provide a procedure to extend optimality to the V-cycle method focusing on the ill-conditioning of the coarse problem.

\subsection{TGM convergence}\label{subsect:TGMconv}
Concerning the validation of the smoothing property, the proof for block-circulant matrices is a slight modification of the one for multilevel scalar-circulant matrices found in \cite{ST}. We report it in full for completeness.

\begin{lemma}[\cite{ST}] \label{lm:smooth}
Let $A_{N}:=\mathcal{A}_{n}(f)$, with $f=[f_{\ell,g}]_{\ell,g=1}^{d}\in\mathcal{M}_{d}$ trigonometric polynomial, $f\geq0$, with 
$f_{j,j}$, $j=1,\ldots,d$, not identically zero, 
and let $V_{N}:=I_{N}-A_{N}/\|f\|_{\infty}$.  If we choose $\alpha_{\rm{post}}$ so
that $\alpha_{\rm{post}}\leq 1/\|f\|_{\infty}$, then relation $(a)$ in Theorem \ref{teoconv} holds true 
(and the best value of $\alpha_{\rm post}$ is $\alpha_{\rm post,best}=1/\|f\|_{\infty}$).
\end{lemma}

\begin{proof}
By setting $V_N = I_N-A_N/\|f\|_{\infty}$, the relation $(a)$ in Theorem \ref{teoconv} is equivalent to writing
\begin{align*}
  \left(I_N-\frac{A_N}{\|f\|_{\infty}}\right)^{H}A_{N}\left(I_N-\frac{A_N}{\|f\|_{\infty}}\right)\leq A_{N}-\alpha_{\rm{post}}A_{N}^2,
\end{align*}
with $\alpha_{\rm{post}}>0$ independent of $n$, that is
\begin{align}\label{condQ}
  \left(I_N-\frac{A_N}{\|f\|_{\infty}}\right)^{2}\leq I_{N}-\alpha_{\rm{post}}A_{N}
\end{align}
By making some algebraic manipulations, the quoted relation can be rewritten as
\begin{align*}
  \frac{A_N^{2}}{\|f\|_{\infty}^{2}} +\left(\alpha_{\rm{post}}-\frac{2}{\|f\|_{\infty}}\right)A_{N}\leq0	
\end{align*}
where the latter is equivalent to requiring that the inequality
\begin{align*}
&\frac{\lambda^{2}}{\|f\|_{\infty}^{2}} +\left(\alpha_{\rm{post}}-\frac{2}{\|f\|_{\infty}}\right)\lambda\leq0
\end{align*}
hold for any eigenvalue $\lambda$ of the Hermitian (positive definite) matrix $A_N$ with $\alpha_{\rm{post}}>0$ independent of $n$.
Since $\|A_{N}\|_{\infty}\leq\|f\|_{\infty}$, and $A_{N}$ is positive definite (the eigenvalues are real and positive), the eigenvalues of $A_{N}$ lie in the range $(0,\|f\|_{\infty}]$.
Therefore a necessary and sufficient condition such that (\ref{condQ}) holds for any $n$ is that $\alpha_{\rm{post}}\leq\frac{1}{\|f\|_{\infty}}$.
\end{proof}
The result of Lemma \ref{lm:smooth} can be easily generalized when
considering both pre-smoothing and post-smoothing as in \cite{AD}.

The following result shows that TGM conditions \eqref{p2f1}, \eqref{p2f3} and \eqref{eqn:commutativity} are sufficient in order to satisfy the approximation property.

\begin{theorem} \label{th:tgmopt}
Let $A_{N}:=\mathcal{A}_{n}(f)$, with $f(\theta)\in\mathcal{M}_{d}$ trigonometric polynomial, $f\geq0$,
and let $p_n^k=\mathcal{A}_{n}(p)(K_{n}^T\otimes I_d)$ be the
projecting operator defined as in Table \ref{CT-projector} and
with $p(\theta)\in\mathcal{M}_{d}$ diagonalizable trigonometric polynomial satisfying conditions $(\ref{p2f1})$, $(\ref{p2f3})$ and (\ref{eqn:commutativity}). Then, there exists a positive value $\gamma$
independent of $n$ such that inequality $(b)$ in Theorem \ref{teoconv} is satisfied.
\end{theorem}
\begin{proof}
In order to prove that there exists $\gamma>0$ independent of $n$ such that for any $x_{N}\in\mathbb{C}^{N}$
\begin{align}\label{condW}
  \min_{y\in\mathbb{C}^{K}}\|x_{N}-p_n^ky\|_{2}^{2}\leq \gamma\|x_{N}\|_{A_{N}}^{2},
\end{align}
we choose a special instance of $y$ in such a way that the previous inequality is
reduced to a matrix inequality in the sense of the partial ordering of the real space of
the Hermitian matrices. For any $x_N\in\mathbb{C}^N$, let $\overline{y}\equiv\overline{y}(x_{N})\in\mathbb{C}^{K}$ be defined as
\begin{align*}
  \overline{y}=[(p_n^k)^{H}p_n^k]^{-1}(p_n^k)^{H}x_{N}.
\end{align*}
We observe that $(p_n^k)^{H}p_n^k$ is invertible, indeed, using the same arguments of Proposition \ref{fhat} with $f=I_{d}$, we have that $(p_n^k)^{H}p_n^k=\mathcal{A}_{k}(\hat{p})$ with $\hat{p}(\theta)=\frac{1}{2}\left(p\left(\frac{\theta}{2}\right)^{H}p\left(\frac{\theta}{2}\right)+
	p\left(\frac{\theta}{2}+\pi\right)^{H}p\left(\frac{\theta}{2}+\pi\right)\right)$ and condition (\ref{p2f3}) ensure that $\hat{p}>0$, that is $\mathcal{A}_{k}(\hat{p})$ is positive definite.\\
Therefore, (\ref{condW}) is implied by
\begin{align*}
  \|x_{N}-p_n^k\overline{y}\|_{2}^{2}\leq \gamma\|x_{N}\|_{A_{N}}^{2},
\end{align*}
where the latter is equivalent to the matrix inequality
\begin{align*}
  W_{N}(p)^{H}W_{N}(p)\leq \gamma A_{N}.
\end{align*}
with $W_{N}(p)=I_{N}-p_n^k[(p_n^k)^{H}p_n^k]^{-1}(p_n^k)^{H}$. Since, by construction, $W_N(p)$ is a Hermitian unitary projector, it holds that $W_N(p)^{H}W_N(p)= W_{N}(p)^{2} = W_N(p)$. As a consequence, the preceding matrix inequality can be rewritten as
\begin{align}\label{rela}
  W_{N}(p)\leq \gamma A_{N}.
\end{align}

Now, using the notation in Table \ref{CT-projector}, the matrix $p_n^k=\mathcal{A}_{n}(p)(K_{n}^{T}\otimes I_{d})$ can be expressed according to
\begin{align*}
  (p_n^k)^{H}=\frac{1}{\sqrt{2}}(F_{k}\otimes I_{d})(I_{n,2}\otimes I_d)\diagi(p(\theta_i^{(n)})^{H})(F_{n}^{H}\otimes I_d),
\end{align*}
and the matrix $(F_{n}^{H}\otimes I_d)W_{N}(p)(F_{n}\otimes I_{d})$ becomes
\begin{align*}
(F_{n}^{H}\otimes I_d)W_{N}(p)(F_{n}\otimes I_{d})&=I_{N}-
\diagi(p(\theta_i^{(n)}))(I_{n,2}^{T}\otimes I_d)\\	
	&\left[\diagik(p(\theta_{i}^{(n)})^{H}p(\theta_{i}^{(n)})+p(\theta_{\tilde{i}}^{(n)})^{H}p(\theta_{\tilde{i}}^{(n)}))\right]^{-1}\\
	&(I_{n,2}\otimes I_d)\diagi(p(\theta_i^{(n)})^{H})
\end{align*}
where $\tilde{i}=i+k$. Now, it is clear that there exists a suitable permutation by rows and columns of $(F_{n}^{H}\otimes I_d)W_{N}(p)(F_{n}\otimes I_{d})$ such that we can obtain a $2d\times 2d$ block-diagonal matrix of the form
\begin{align*}
I_{N}-\diagik\left[\begin{array}{c} 
p(\theta_{i}^{(n)}) \\ p(\theta_{\tilde{i}}^{(n)})
\end{array}\right]
\left[\begin{array}{c} 
(p(\theta_{i}^{(n)})^{H}p(\theta_{i}^{(n)})+p(\theta_{\tilde{i}}^{(n)})^{H}p(\theta_{\tilde{i}}^{(n)}))^{-1}
\end{array}\right]
\left[\begin{array}{cc} 
p(\theta_{i}^{(n)})^{H} & p(\theta_{\tilde{i}}^{(n)})^{H}
\end{array}\right].
\end{align*}
Therefore, by considering the same permutation by rows and columns of $(F_{n}^{H}\otimes I_d)A_{N}(F_{n}\otimes I_{d})=\Delta_{N}(f)$, condition (\ref{rela}) is equivalent to requiring that there exists $\gamma>0$ independent of $n$ such that, $\forall j=0,\ldots,k-1$
\begin{align*}
&I_{2d}-\left[\begin{array}{c} 
p(\theta_{i}^{(n)}) \\ p(\theta_{\tilde{i}}^{(n)})
\end{array}\right]
\left[\begin{array}{c} 
(p(\theta_{i}^{(n)})^{H}p(\theta_{i}^{(n)})+p(\theta_{\tilde{i}}^{(n)})^{H}p(\theta_{\tilde{i}}^{(n)}))^{-1}
\end{array}\right]
\left[\begin{array}{cc} 
p(\theta_{i}^{(n)})^{H} & p(\theta_{\tilde{i}}^{(n)})^{H}
\end{array}\right] \\
&\leq \gamma
\left[\begin{array}{cc} 
f(\theta_{i}^{(n)}) & \\
& f(\theta_{\tilde{i}}^{(n)})
\end{array}\right].
\end{align*}
We define the set $H=\{\eta|\eta\in\{\theta,(\theta+\pi)\mod 2\pi\},\mbox{ where }\theta\in\Theta_0\}$ . Due of the continuity of $p$ and $f$ it is clear that the preceding set of inequalities can be reduced to requiring that a unique inequality of the form
\begin{align*}
I_{2d}-\left[\begin{array}{c} 
p(\theta) \\ p(\theta+\pi)
\end{array}\right]
\left[\begin{array}{c} 
(p(\theta)^{H}p(\theta)+p(\theta+\pi)^{H}p(\theta+\pi))^{-1}
\end{array}\right]
\left[\begin{array}{cc} 
p(\theta)^{H} & p(\theta+\pi)^{H}
\end{array}\right]\leq \gamma
\left[\begin{array}{cc} 
f(\theta) & \\
& f(\theta+\pi)
\end{array}\right],
\end{align*}
holds for all $\theta\in[0,2\pi)\backslash H$. Let us define $q(\theta)=(p(\theta)^{H}p(\theta)+p(\theta+\pi)^{H}p(\theta+\pi))^{-1}$. By simple computations, using condition (\ref{eqn:commutativity}) and Remark \ref{remark:commutativity} the previous inequality becomes 
\begin{align}\label{eqn:after_commutativity}
\left[\begin{array}{c} 
q(\theta)
\end{array}\right]
\left[\begin{array}{cc} 
p(\theta+\pi)^{H}p(\theta+\pi) & p(\theta)p(\theta+\pi)^{H} \\
p(\theta+\pi)p(\theta)^{H} & p(\theta)^{H}p(\theta)
\end{array}\right]\leq \gamma
\left[\begin{array}{cc} 
f(\theta) & \\
& f(\theta+\pi)
\end{array}\right].
\end{align}

Let us define the matrix-valued function
\begin{align*}
R(\theta)=\left[\begin{array}{cc} 
f(\theta) & \\
& f(\theta+\pi)
\end{array}\right]^{-\frac{1}{2}}
\left[\begin{array}{c} 
q(\theta)
\end{array}\right]
\left[\begin{array}{cc} 
p(\theta+\pi)^{H}p(\theta+\pi) & p(\theta)p(\theta+\pi)^{H} \\
p(\theta+\pi)p(\theta)^{H} & p(\theta)^{H}p(\theta)
\end{array}\right]
\left[\begin{array}{cc} 
f(\theta) & \\
& f(\theta+\pi)
\end{array}\right]^{-\frac{1}{2}}.
\end{align*}
By the Sylvester inertia law \cite{GV}, relation \eqref{eqn:after_commutativity} is satisfied if 
\begin{align}\label{eqn:inequality_on_R}
R(\theta)\leq \gamma I_{2d} 
\end{align}
is satisfied, which is equivalent to show that the matrix-valued function $R(\theta)$ is uniformly bounded in the spectral norm. Using again the commutativity hypothesis \eqref{eqn:commutativity}, we can write $R(\theta)$ as
\begin{align*}
R(\theta)=
\left[\begin{array}{cc} 
f^{-\frac{1}{2}}(\theta)p(\theta+\pi)^{H}q(\theta)p(\theta+\pi)f^{-\frac{1}{2}}(\theta) & f^{-\frac{1}{2}}(\theta)p(\theta+\pi)^{H}q(\theta)p(\theta)f^{-\frac{1}{2}}(\theta+\pi) \\
f^{-\frac{1}{2}}(\theta+\pi)p(\theta)^{H}q(\theta)p(\theta+\pi)f^{-\frac{1}{2}}(\theta) & f^{-\frac{1}{2}}(\theta+\pi)p(\theta)^{H}q(\theta)p(\theta)f^{-\frac{1}{2}}(\theta+\pi)
\end{array}\right].
\end{align*}

We prove that $R(\theta)$ is uniformly bounded in the spectral norm by proving that all its components are uniformly bounded in the norm $|\cdot|_1$. For all $\theta\in[0,2\pi)\backslash H$, we can write
\begin{align*}
\left|R_{1,1}(\theta)\right|_1=\left|f^{-\frac{1}{2}}(\theta)p(\theta+\pi)^{H}q(\theta)p(\theta+\pi)f^{-\frac{1}{2}}(\theta)\right|_1\le\left|f^{-\frac{1}{2}}(\theta)p(\theta+\pi)^{H}\right|_1\left|q(\theta)\right|_1\left|p(\theta+\pi)f^{-\frac{1}{2}}(\theta)\right|_1.
\end{align*}

Noticing that
$$
\left|p(\theta+\pi)f^{-\frac{1}{2}}(\theta)\right|_1=\left|\left(f^{-\frac{1}{2}}(\theta)p(\theta+\pi)^{H}\right)^H\right|_1=\left|f^{-\frac{1}{2}}(\theta)p(\theta+\pi)^{H}\right|_1
$$
and using conditions \eqref{p2f1} and \eqref{p2f3}, we can find $\delta$ such that $|R_{1,1}(\theta)|_1<\delta$ for all $\theta\in[0,2\pi)\backslash H$.

The uniform boundedness of the other components of $R(\theta)$ can be proven in an analogous way, recalling that if $\theta$ belongs to $\Theta_0$, then $f$ is nonsingular in $\theta+\pi$. This implies that the matrix-valued function $R(\theta)$ is uniformly bounded in the 1-norm. Since the matrix dimension of $R(\theta)$ is fixed for all $\theta$ and equal to $2d$, the equivalence between the 1-norm and the spectral norm lets us conclude the proof.
\end{proof}

\subsection{MGM convergence and optimality}\label{subsect:MGM}
In the current subsection we consider a problem of Laplacian type i.e.  $\mathcal{A}_n(f)\ge 0$ generated by a trigonometric polynomial $f:Q\rightarrow\mathcal{M}_d$, $f\geq0$, that has a nonnegative  minimal eigenvalue function $\lambda_{\rm min}(f)$ with a unique zero in the origin of order two.

In order to select a projector $p_n^k$ that ensures the convergence and optimality of the multigrid procedure applied to $\mathcal{A}_n(f)$, we study the quantity

\begin{equation*}
\kappa(\hat{f}_j)=\frac{\|\lambda_{\rm max}(\hat{f}_j)\|_{\infty}}{\left.\lambda''_{\rm min}(\hat{f}_j)\right|_{0}},
\end{equation*} 
which gives an estimate of the ill-conditioning of the coarse problem at level $j$. Indeed the conditioning of the matrix $\mathcal{A}_{n_j}(\hat{f}_j)$ depends on $\|\lambda_{\rm max}(\hat{f}_j)\|_{\infty}$ and $\left.\lambda''_{\rm min}(\hat{f}_j)\right|_{0}$,  which measure the magnitude of the maximum eigenvalue function $\lambda_{\rm max}(\hat{f}_j)$ and how flat the minimal eigenvalue function is around the origin, respectively. 

We select a class of projectors $p_{n(j)}^{k(j)}(z)=\mathcal{A}_{n(j)}(p_z)(K^T_{n(j)}\otimes I_d)$ according to the theoretical analysis of Section \ref{sect:project_circ} with $p_z(\cdot)$ of form

\begin{equation}\label{projector_pz}
p_z(\theta)=(1+\cos \theta)\left(I_d+\frac{z-1}{d}ee^T\right), \quad z > 0
\end{equation} 
where $e$ is the vector of all ones of length $d$.
Note that

\begin{equation*}
p_z(\theta)=F_d\begin{bmatrix}
z+z\cos\theta& & &  \\
             &1+\cos\theta & &  \\
        & & \ddots&  \\
        & & &1+\cos\theta  \\
\end{bmatrix} F_d^H
\end{equation*} 
hence the eigenvalue functions of $p_z(\cdot)$ have a zero at $\pi$ of order two for all $z > 0$, which is the desirable property for Condition (\ref{p2f1}). Moreover, the matrix-valued function $p_z(\cdot)$ trivially satisfies Condition (\ref{eqn:commutativity}), since its eigenvector functions are constant.

In the following section we will study the conditioning $\kappa(\hat{f}_{z,j})$, where  $\hat{f}_{z,j}$ is the generating function at level $j$ obtained using $p_z(\cdot)$. In particular we will look for a $z>0$ such that 

\begin{equation}\label{formula:limit_lambda_j}
\lim_{j\to \infty}\left.\lambda''_{\rm min}(\hat{f}_{z,j})\right|_{0}>0
\end{equation}  
that guarantees that the behaviour of the minimal eigenvalue function around the origin remains unchanged at the coarser levels. 

\section{Extension to 2D case}\label{sect:2D}
In the following we show how it is possible extend the MGM convergence results in the multidimensional setting.
Let ${\bf n}:=(n_1,\ldots,n_\Bbbk)$ be a multi-index
in $\mathbb N^\Bbbk$ and set $N(d,\textbf{n}):=d\prod_{i=1}^\Bbbk n_i$. In particular we show how to generalize projector $p_n^k$ for the $\Bbbk-$level block-circulant matrix $A_N=C_{\textbf{n} }(f)$ of dimension $N(d,\textbf{n})$ generated by a multilevel block-circulant trigonometric polynomial $f$. For a complete discussion on the multi-index notation, see \cite{GS}.

\begin{Definition} \label{ multivariate_and_matrix}
A matrix-valued multivariate trigonometric polynomial is a function $f:Q^\Bbbk\to \C^{d \times d}$, $\Bbbk,d>1$,  written as a finite linear combination of the Fourier frequencies $\{{\rm e}^{-{\imath}\left\langle {\bf j},\boldsymbol{\theta}\right\rangle}\, : \textbf{j}\in \Z^\Bbbk\}$ or, equivalently, for all $l,m=1,\dots, s$, its $(l,m)$th component $f_{lm}:Q^\Bbbk\to \C$ is a scalar multivariate trigonometric polynomial of degree $\textbf{r}_{lm}$. 
The degree of $f$ is a positive $\Bbbk$-index $\boldsymbol{{\rm r}}$ defined as
\[{\boldsymbol{{\rm r}}}=\underset{m=1,\dots, d}{\max_{l=1,\dots, d}} \textbf{r}_{lm}. \]
 Thus $f$ can be written as the Fourier sum
 \begin{equation} \label{f_multi}
 f(\boldsymbol{\theta})=\sum_{\textbf{j}=-\boldsymbol{{\rm r}}}^{ \boldsymbol{{\rm r}}} a_\textbf{j} {\rm e}^{\left\langle {\bf j},\boldsymbol{\theta}\right\rangle}, 
 \end{equation}
 where the Fourier coefficients of $f$ are given by 
 \begin{align}
  a_{\bf j}:=\frac1{(2\pi)^\Bbbk}\int_{Q^\Bbbk}f(\boldsymbol{\theta}){\rm e}^{-{\imath}\left\langle {\bf j},\boldsymbol{\theta}\right\rangle}\ {\rm d}\boldsymbol{\theta}\in\C^{d\times d},
  \qquad {\bf j}=(j_1,\ldots,j_\Bbbk)\in\mathbb Z^\Bbbk,\ \ \ 
\end{align}
\end{Definition}
where $\left\langle { \bf j},\boldsymbol{\theta}\right\rangle=\sum_{t=1}^\Bbbk j_t\theta_t$ and the integrals in \eqref{fhat} are computed componentwise. 
If $f$ is defined ad in (\ref{f_multi}), then the ${\bf n}$th multilevel block-circulant matrix associated with $f$ is the matrix of order $N(d,\textbf{n})$ given by
\begin{equation}\label{Circ_multi_block}
C_{\bf n}(f) =\sum_{\bf j=-(\bf n-\bf e)}^{\bf n-\bf e} Z_{n_1}^{j_1} \otimes \cdots\otimes Z_{n_\Bbbk}^{j_\Bbbk}\otimes a_{\bf j},
\end{equation}
where $\textbf{e}=(1,\ldots,1)\in\mathbb{N}^\Bbbk, \,\textbf{j}=(j_1,\ldots,j_\Bbbk)\in\mathbb{N}^\Bbbk$ and $ Z^{j_{\xi}}_{n_{\xi}}$  is the $n_{\xi} \times n_{\xi}$ matrix whose $(i,h)$th entry equals 1 if $(i-h)$ mod $n_{\xi}=j_{\xi}$ and $0$ otherwise.

Analogously to the scalar case, we want to construct the projectors from an arbitrary multilevel block circulant matrix $C_{\textbf{n}}(p)$, with $p$ multivariate matrix-valued trigonometric polynomial of degree $\textbf{c}$  independent of $\textbf{n}$.
 Hence we define the projector 
 \begin{equation}
  p_{\textbf{n}}^{\textbf{k}}=C_{\textbf{n}}(p) \left(K^T_{\textbf{n}}\otimes I_d\right),
 \end{equation} 
  where $K_{\textbf{n}}$ the $N(1,\textbf{n}) \times \frac{N(1,\textbf{n})}{2^\Bbbk}$ matrix defined by $K_{\textbf{n}}=K_{n_1} \otimes K_{n_2} \otimes \dots \otimes K_{n_\Bbbk}$ and $C_{\textbf{n}}(p)$ is a multilevel block-circulant matrix generated by $p$.


In the next section we will see how an analogous procedure can be applied to multilevel Teoplitz structures. In particular, in Subsections \ref{subsection:FEM2D} and \ref{subsection:staggeredDG} we apply the V-cycle procedure to bilevel Toeplitz matrices ($\Bbbk = 2$), assuming $\textbf{n}=(n,n)$.

\section{Numerical Examples}\label{sect:experiments}

In the current section we give numerical evidence of the results proven in Section \ref{sect:proof_circ}. We will deal with general Toeplitz matrices generated by a matrix-valued trigonometric polynomial, instead of block-circulant matrices. We expect that the theoretical results of Section \ref{sect:proof_circ} still hold, since the analysis for Toeplitz matrices is an algebraic generalization of the Local Fourier Analysis of multigrid methods \cite{D}.

As far as the choice of the right-hand side is concerned, we impose that the solution $x$ of the linear system $T_n({f})x = b$ is a uniform sampling of the sine function on $[0,\pi]$. We compute the right-hand side $b$ as $b = T_n({f})x$.

The structure of the projector slightly changes for block-Toeplitz matrices, in order to preserve the structure at coarser levels. The dimension of the problem at level $t$ becomes $N=nd$, with $n$ of the form $2^t-1$. The cutting matrix $K_{n}$ takes the form 
$$
K_{n} = \left[\begin{array}{ccccccccc}
0 & 1 & 0 & & & & &\\
&   & 0 & 1 & 0 & & & & \\
&   &   &   & \ddots & \ddots & \ddots & & \\
&   &   &   &        &        & 0      & 1 & 0		
\end{array}\right]_{\frac{n-1}{2}\times n}
$$
and, for a matrix-valued trigonometric polynomial $p$, the projector is
\begin{equation}\label{multi_proj}
p_n^k=T_{n}(p)\left(K_{n}^T\otimes I_d\right).
\end{equation}

In Subsection \ref{subsection:FEM} we present strategies for an implementation of both TGM and MGM for $\mathbb{Q}_{\deg}$ Lagrangian FEM stiffness matrices for the second order elliptic differential problem on $[0,1]$.

In Subsection \ref{subsection:FEM2D} we consider the two-dimensional problem, i.e. we study multigrid methods for the $\mathbb{Q}_{\deg}$ Lagrangian FEM stiffness matrices for the second order elliptic differential problem on the unit square.

In Subsection \ref{subsection:staggeredDG}, we apply our multigrid strategies to the matrices stemming from the discretization by staggered discontinuous Galerkin methods of the incompressible Navier-Stokes equations.

Apart from the first example, we will use the Gauss-Seidel method as a smoother. The method damps the high frequencies, which makes it a suitable smoother for our problems.

\textcolor{black}{
In Subsection \ref{subsection:FEM} we also present results with the relaxed Jacobi method as a smoother. We state the following remarks to show how to choose the relaxation parameter $\omega$ for the applicability of Lemma \ref{lm:smooth} to the Jacobi method.
}
\textcolor{black}{
\begin{remark}\label{rmk:richardson}
	For the relaxed Richardson method with iteration matrix $V_{n}:=I_{N}-\omega T_n({f})$, we follow the proof of Lemma \ref{lm:smooth} and we see that, in order to satisfy relation $(a)$ in Theorem \ref{teoconv}, there should exist $\alpha_{\rm post}>0$ such that 
	$$
	\omega^2\|f\|_\infty^2-2\omega\|f\|_\infty+\alpha_{\rm post}\|f\|_\infty \le 0,
	$$
	from which we can write
	$$
	\alpha_{\rm post}\le \frac{-\omega^2\|f\|_\infty^2+2\omega\|f\|_\infty}{\|f\|_\infty}.
	$$
	For the existence of such a $\alpha_{\rm post}>0$, the right-hand side should be greater than 0, and this leads to the following quadratic inequality:
	$$
	-\omega^2\|f\|_\infty^2+2\omega\|f\|_\infty \ge 0,
	$$
	which has solution
	$$
	0\le \omega\le \frac{2}{\|f\|_\infty}.
	$$
\end{remark}
}
\textcolor{black}{
\begin{remark}
The iteration matrix of the relaxed Jacobi method is $V_{n}:=I_{N}-\omega D_n^{-1}T_n({f})$, where $D_n$ is a diagonal matrix with the same diagonal as $T_n({f})$. We define the matrix $\tilde{D}_n:= \min_{j=1,\dots,d}{\left(a_0^{(j,j)}\right)}I_N$ and we notice that $\tilde{D}_n^{-1}\ge D_n^{-1}$. Applying to the matrix $I_{N}-\omega \tilde{D}_n^{-1}T_n({f})$ the same idea that we used for the Richardson method in Remark \ref{rmk:richardson}, we obtain that relation $(a)$ in Theorem \ref{teoconv} is satisfied if $\omega$ verifies the following inequality:
\begin{equation}\label{eqn:jacobi_omega}
0\le \omega\le \frac{2\min_{j=1,\dots,d}{\left(a_0^{(j,j)}\right)}}{||f||_\infty}.
\end{equation}
\end{remark}
}

\subsection{$\mathbb{Q}_{\deg}$ Lagrangian FEM stiffness matrices: the 1D case}\label{subsection:FEM}

Consider the $\mathbb{Q}_{\deg}$ Lagrangian Finite Element approximation (FEM) of the second order elliptic differential problem
\begin{equation}\label{second_order}
\begin{cases}
- u''(x)=\phi(x), & \text{ on }(0,1),\\
\hfill u(0)=u(1)=0
\end{cases}.
\end{equation}

The resulting stiffness matrix of size  $(\deg\cdot n-1)\times (\deg\cdot n-1)$ is $nK_n^{(\deg)}$, where 
$K_n^{(\deg)}$ is a block-Toeplitz matrix
\begin{align}
K_n^{(\deg)}=T_n({f})_-,\nonumber
\end{align}
with the subscript $-$ denoting that the last row and column of $T_n({f})$ are removed. This is because of the homogeneous boundary conditions.

The construction of the matrix and the symbol is given in~\cite{qp}.
The $\deg \times \deg$ matrix-valued generating function of $T_n(f)$ is 
\begin{align}
{f}(\theta)=a_0+a_1e^{\imath \theta}+a_1^{\mathrm{T}}e^{-\imath \theta}\nonumber
\end{align}
In the following we want to apply the MGM strategy to the matrix $\mathcal{A}_N=T_n(f)$, for different choices of $\deg$. Indeed there exist $n$ points $\theta_i^{(n)}$ and a unitary transform $Q_n$ such that

\begin{equation}
T_n(f)=Q_n \diagi(f(\theta_i^{(n)}))Q_{n}^{H}.
\end{equation}

Moreover in \cite{qp} authors prove that there exists a constant $c_{\deg}>0$ such that, for all $\theta$

\begin{equation*}
c_{\deg}(2-2\cos \theta)\le \lambda_{\rm min}(f(\theta))\le 2-2\cos\theta,
\end{equation*}
which guarantees that $\lambda_{\rm min}(f(\theta))$ has a zero of order 2 at the origin.

\begin{flushleft}
\textbf{TGM in the $\deg=2$ setting}
\end{flushleft}

In Example 1 of~\cite{qp} the case for $\deg=2$ is presented. In particular, the explicit expressions of $a_0$, $a_1$ are given by
\begin{equation}\label{eqn:a0_a1_1D_Q2}
a_0 = \frac{1}{3}\begin{bmatrix}
16 & -8 \\
-8 & 14
\end{bmatrix}, \quad 
a_1 = \frac{1}{3}\begin{bmatrix}
 0 & -8 \\
 0 & 1
\end{bmatrix}.
\end{equation}
Moreover, it is possible to diagonalize $f$ as 
\[
f(\theta) = U(\theta)
\begin{bmatrix}
\lambda_1(f(\theta)) & \\
& \lambda_2(f(\theta))
\end{bmatrix}U^H(\theta),
\]
where the eigenvalue functions $\lambda_1(f(\theta)), \lambda_2(f(\theta))$ of $f$ are given explicitly by
\begin{align}
\lambda_1(f(\theta))&=5+\frac{1}{3}\cos(\theta)-\frac{1}{3}\sqrt{129+126\cos(\theta)+\cos^2(\theta)},\nonumber\\
\lambda_2(f(\theta))&=5+\frac{1}{3}\cos(\theta)+\frac{1}{3}\sqrt{129+126\cos(\theta)+\cos^2(\theta)} \nonumber
\end{align}
and $U:Q\rightarrow \mathcal{M}_2$ is the matrix-valued function containing the eigenvectors of $f$.

The hypotheses requested in Section \ref{sect:project_circ} that ensure the convergence and optimality of the TGM for $T_n(f)$ are satisfied using $p_z$ in the construction of the projector.

However, we notice that $p_z$ has and additional property. It can be shown by direct computation that $f(0)p_z(0)=p_z(0)f(0)$ for every choice of $z>0$. This implies that $f(0)$ and $p_z(0)$ are simultaneously diagonalized by the same unitary transform. Therefore, we can control the ill-conditioning of the coarser problems in the subspace associated to $\theta = 0$ by taking different values of $z$. This will be useful for the study of the V-cycle method. 

Now we implement a two grid procedure for $T_n(f)$ and we study the number of iterations that the method requires to reach the desired tolerance varying $n$ and for different choices of $z$.

\textcolor{black}{
In order to find the relaxation parameters for the Jacobi method we should compute the quantities in inequality (\ref{eqn:jacobi_omega}). We see from formula (\ref{eqn:a0_a1_1D_Q2}) that $\min_{j=1,\dots,s}\left(a_0^{(j,j)}\right)$ is equal to $14/3$. 
For the computation of the quantity ${||f||_\infty}={\max}_{\theta\in Q}\|f(\theta)\|_\infty$ we can write 
$$
||f||_\infty = \frac{1}{3} \max \left(\max_{\theta \in Q}(16+|8+8e^{-\imath\theta}|), \max_{\theta \in Q} (|8+8e^{\imath\theta}|+14+2\cos(\theta))\right) = \frac{32}{3}.
$$
So, according to inequality (\ref{eqn:jacobi_omega}), our Jacobi relaxation parameter $\omega$ should be smaller than or equal to 7/8. In order to damp the error both in the middle and in the high frequencies, we take a different parameter for the pre-smoother and the post-smoother. For the pre-smoother, we take the greatest admissible value, $\omega_{\rm pre} = 7/8$, and for the post-smoother we take $\omega_{\rm post} = 2\omega_{\rm pre}/3 = 7/12$.
}

In Tables \ref{tab:TGM_Q2_Jac}-\ref{tab:TGM_Q2_GS} we report for $z = 1,\dots,5$ the number of iterations needed for achieving the tolerance $\epsilon = 10^{-7}$ when increasing the matrix size and using $p_z$ in the construction of the projector and with two different smoothers.
Table \ref{tab:TGM_Q2_Jac} shows the results using as pre- and post-smoother one iteration of the Jacobi method with relaxation parameters $\omega_{\rm pre} = 7/8$ and $\omega_{\rm post} = 7/12$. Table \ref{tab:TGM_Q2_GS} shows the results using as pre- and post-smoother one iteration of the Gauss-Seidel method with $\omega_{\rm pre,post} = 1$.

As expected, in both cases we can observe that for all $z = 1,\dots,5$ the number of iterations needed for the TGM convergence remains almost constant, when increasing the size $N$, confirming the optimality of the method for every choice of $z$.

\begin{table}[htb]
	\begin{center}
		\begin{tabular}{cccccccccc}
			\hline
			$t$&$n=2^t-1$& N=$2 n$&$z=1$&$z=2$&$z=3$&$z=4$&$z=5$\\
			\hline
			3&	  7&   14& 28&28 &28&28&28\\
			4&	 15&   30& 32&32 &32&32&32\\
			5&	 31&   62& 33&33 &33&33&33\\
			6&	 63&  126& 33&33 &33&33&33\\
			7& 	127&  254& 33&33 &33&33&33\\
			8& 	255&  510& 33&33 &33&33&33\\
			9& 	511& 1022& 33&33 &33&33&33\\
			10& 1023& 2046&33&33 &33&33&33\\
			11& 2047& 4094&33&33 &33&33&33\\
			\hline
		\end{tabular}
	\end{center}
	\caption{Number of iterations for the Two-Grid method applied to the $\mathbb{Q}_2$ Lagrangian FEM Stiffness matrix, using as pre- and post-smoother one iteration of Jacobi method with $\omega_{\rm pre}=7/8$, $\omega_{\rm post}=7/12$ and tolerance $\epsilon = 10^{-7}$.}
	\label{tab:TGM_Q2_Jac}
\end{table}

\begin{table}[htb]
	\begin{center}
		\begin{tabular}{cccccccccc}
			\hline
			$t$&$n=2^t-1$& N=$2 n$&$z=1$&$z=2$&$z=3$&$z=4$&$z=5$\\
			\hline
			3&	  7&   14& 15& 15 &15&15&15\\
			4&	 15&   30& 15& 15 &15&15&15\\
			5&	 31&   62& 15& 15 &15&15&15\\
			6&	 63&  126& 15& 15 &15&15&15\\
			7& 	127&  254& 15& 15 &15&15&15\\
			8& 	255&  510& 15& 15 &15&15&15\\
			9& 	511& 1022& 15& 15 &15&15&15\\
			10& 1023& 2046&15& 15 &15&15&15\\
			11& 2047& 4094&15& 15 &15&15&15\\
			\hline
		\end{tabular}
	\end{center}
	\caption{Number of iterations for the Two-Grid method applied to the $\mathbb{Q}_2$ Lagrangian FEM Stiffness matrix, using as pre- and post-smoother one iteration of Gauss-Seidel method with $\omega_{\rm pre,post}=1$ and tolerance $\epsilon = 10^{-7}$.}
	\label{tab:TGM_Q2_GS}
\end{table}

\begin{flushleft}
	\textbf{MGM in the $\deg=2$ setting}	
\end{flushleft}

In order to maintain the optimality of the iterations also for the MGM  
we should look for the best choice of the parameter $z$ such that the behaviour of $\lambda_{\rm min}(\hat{f}_{z,j})$ around the origin remains unchanged at the coarser levels, that is, for different choices of $z$, we check if  $\lambda_{\rm min}(\hat{f}_{z,j})$ satisfies condition (\ref{formula:limit_lambda_j}).

 By direct computation, we derive the formula
\[
\left.\lambda''_{\rm min}(\hat{f}_{z,j})\right|_{0} = \left(\frac{z^2}{2}\right)^j.
\]
The latter implies that for values of $z$ smaller than $\sqrt{2}$, the quantity $\left.\lambda''_{\rm min}(\hat{f}_{z,j})\right|_{0}$ tends to zero as $j$ tends to $\infty$. This suggests that for $z<\sqrt{2}$ the conditioning becomes worse as the levels get coarser.
This is numerically confirmed in Table \ref{tab:condition_number_q2} where  the condition numbers $\kappa(\hat{f}_{z,j})$ are listed for $z=1,2,3,4$ and $j=1,2,3,4$. Therefore we should avoid the choice $p_{n(j)}^{k(j)}(1)$ as projector.

Indeed,  Tables \ref{tab:MGM_Q2_Jac}-\ref{tab:MGM_Q2_GS} highlight that  the number of iterations needed for the MGM convergence, with the desired tolerance, depends on the matrix size with $z=1$, whereas it remains almost constant for $z>\sqrt{2}$ as $n$ increases.

\begin{table}
	\begin{center}
		\begin{tabular}{cccccccc}
			\hline
			$j$&$\kappa(\hat{f}_{1,j})$& $\kappa(\hat{f}_{2,j})$&$\kappa(\hat{f}_{3,j})$&$\kappa(\hat{f}_{4,j})$\\
			\hline
			1&	 43&   11& 4.7& 4.7\\
			2&	171&   11& 4.7& 4.7\\
			3&	683&   11& 4.7& 4.7\\
			4& 2731&   11& 4.7& 4.7\\
			\hline
		\end{tabular}
	\end{center}
	\caption{Condition numbers of $\hat{f}_{z,j}$ for $z = 1,2,3,4$ and $j = 1,2,3,4$.}
	\label{tab:condition_number_q2}
\end{table}

\begin{table}[htb]
	\begin{center}
		\begin{tabular}{cccccccccc}
			\hline
			$t$&$n=2^t-1$& N=$2 n$&$z=1$&$z=2$&$z=3$&$z=4$&$z=5$\\
			\hline
			3&	  7& 	  14&    28& 28 &28&28&28\\
			4&	 15& 	  30&    65& 34 &34&35&39\\
			5&	 31& 	  62&   155& 36 &34&35&38\\
			6&	 63& 	 126&   407& 39 &34&35&39\\
			7& 	127& 	 254&  1144& 42 &34&35&38\\
			8& 	255& 	 510&  3365& 45 &35&35&37\\
			9& 	511& 	1022& 4000+& 48 &35&35&37\\
			10& 1023&	2046& 4000+& 50 &35&35&37\\
			11& 2047&	4094& 4000+& 52 &35&35&38\\ 
			12& 4095&	8190& 4000+& 54 &35&36&38\\
			13& 8191&  16382& 4000+& 55 &35&36&38\\ 
			\hline
		\end{tabular}
	\end{center}
	\caption{Number of iterations for the V-cycle method applied to the $\mathbb{Q}_2$ Lagrangian FEM Stiffness matrix, pre- and post-smoother 1 iteration of Jacobi with $\omega_{\rm pre}=7/8$ and $\omega_{\rm post}=7/12$, tolerance $\epsilon = 10^{-7}$.}
	\label{tab:MGM_Q2_Jac}
\end{table}

\begin{table}[htb]
	\begin{center}
		\begin{tabular}{cccccccccc}
			\hline
			$t$&$n=2^t-1$& N=$2 n$&$z=1$&$z=2$&$z=3$&$z=4$&$z=5$\\
			\hline
			3&	  7&   14&  15 &15&15&15&15 \\
			4&	 15&   30&  28 &19&16&17&18 \\
			5&	 31&   62&  67 &21&19&20&21 \\
			6&	 63&  126& 171 &23&21&21&23 \\
			7& 	127&  254& 467 &26&22&23&26 \\
			8& 	255&  510&1343 &29&23&26&28 \\
			9& 	511& 1022&3992 &31&24&28&30 \\
			10& 1023&2046&4000+&33&27&29&32 \\
			11& 2047&4094&4000+&35&28&30&33 \\
			12& 4095&8190&4000+&36&29&31&34 \\
			13& 8191&16382&4000+&38&29&32&34 \\
			\hline
		\end{tabular}
	\end{center}
	\caption{Number of iterations for the V-cycle method applied to the $\mathbb{Q}_2$ Lagrangian FEM Stiffness matrix, pre- and post-smoother 1 iteration of Gauss-Seidel with $\omega_{\rm pre,post}=1$, tolerance $\epsilon = 10^{-7}$.}
	\label{tab:MGM_Q2_GS}
\end{table}

\begin{flushleft}
\textbf{TGM and MGM in the $\deg>2$ setting}  
\end{flushleft}

We implemented the analogous TGM for polynomial degrees 3 and 4. From Tables \ref{tab:TGM_Q3}-\ref{tab:TGM_Q4} we see that the number of iterations to achieve the desired tolerance still remains constant as the matrix size increases. However, we notice that this constant depends on the polynomial degree $\deg$. Achieving optimality from this point of view is beyond the scope of this paper.

The analysis on the condition number that we exploited for $\deg=2$ can be repeated assuming that Conjecture \ref{conj} (numerically verified for $\deg=3,4$) holds. 

\begin{conjecture}\label{conj}
For every $\deg>0$, $j>0$, $z>0$ there exists $c_{z,\deg}>0$ such that the following equality holds
\begin{equation*}
\left.\lambda''_{\rm min}(\hat{f}_{z,j})\right|_{0} =c_{z,\deg}\left(\frac{z^2}{2}\right)^{j}.
\end{equation*}
\end{conjecture} 

The numerical experiments confirm the theoretical analysis deriving from the previous conjecture, as we can see from the number of iterations obtained for $\deg=3,4$ in Tables \ref{tab:MGM_Q3}-\ref{tab:MGM_Q4}. Indeed, analogously to the case $\deg=2$, we observe that we should avoid to take $z=1$, for which $\left.\lambda''_{\rm min}(\hat{f}_{z,j})\right|_{0}$ tends to 0 as $j$ tends to $\infty$.

\begin{table}[htb]
	\begin{center}
		\begin{tabular}{cccccccccc}
			\hline
			$t$&$n=2^t-1$& N=$3 n$&$z=1$&$z=2$&$z=3$&$z=4$&$z=5$\\
			\hline
			3&	  7&   21& 34& 34 &34&34&34\\
			4&	 15&   45& 38& 38 &38&38&38\\
			5&	 31&   93& 38& 38 &38&38&38\\
			6&	 63&  189& 38& 38 &38&38&38\\
			7& 	127&  381& 38& 38 &38&38&38\\
			8& 	255&  765& 38& 38 &38&38&38\\
			9& 	511& 1533& 38& 38 &38&38&38\\
			10& 1023&3069& 38& 38 &38&38&38\\
			11& 2047&6141& 38& 38 &38&38&38\\
			\hline
		\end{tabular}
	\end{center}
	\caption{Number of iterations for the Two-Grid method applied to the $\mathbb{Q}_3$ Lagrangian FEM Stiffness matrix, pre- and post-smoother 1 iteration of Gauss-Seidel with $\omega_{\rm pre,post}=1$, tolerance $\epsilon = 10^{-7}$.}
	\label{tab:TGM_Q3}
\end{table}

\begin{table}[htb]
	\begin{center}
		\begin{tabular}{cccccccccc}
			\hline
			$t$&$n=2^t-1$& N=$4 n$&$z=1$&$z=2$&$z=3$&$z=4$&$z=5$\\
			\hline
			3&	  7&   28& 81& 81 &81&81&81\\
			4&	 15&   60& 86& 86 &86&86&86\\
			5&	 31&  124& 87& 87 &87&87&87\\
			6&	 63&  252& 87& 87 &87&87&87\\
			7& 	127&  508& 87& 87 &87&87&87\\
			8& 	255& 1020& 87& 87 &87&87&87\\
			9& 	511& 2044& 87& 87 &87&87&87\\
			10& 1023&4092& 87& 87 &87&87&87\\
			11& 2047&8188& 87& 87 &87&87&87\\
			\hline
		\end{tabular}
	\end{center}
	\caption{Number of iterations for the Two-Grid method applied to the $\mathbb{Q}_4$ Lagrangian FEM Stiffness matrix, pre- and post-smoother 1 iteration of Gauss-Seidel with $\omega_{\rm pre,post}=1$, tolerance $\epsilon = 10^{-7}$.}
	\label{tab:TGM_Q4}
\end{table}

\begin{table}[htb]
	\begin{center}
		\begin{tabular}{cccccccccc}
			\hline
			$t$&$n=2^t-1$& N=$3 n$&$z=1$&$z=2$&$z=3$&$z=4$&$z=5$\\
			\hline
			3&	  7&    21&	  34& 34 &34&34&34\\
			4&	 15&    45&   79& 42 &37&39&40\\
			5&	 31&    93&  175& 44 &39&41&42\\
			6&	 63&   189&  436& 47 &41&42&43\\
			7& 	127&   381& 1180& 51 &43&44&46\\
			8& 	255&   765& 3375& 55 &44&47&50\\
			9& 	511&  1533& 4000+& 59 &45&51&52\\
			10& 1023& 3069& 4000+& 63 &47&52&54\\
			11& 2047& 6141& 4000+& 66 &50&54&56\\
			12& 4095&12285& 4000+& 69 &53&55&57\\
			13& 8191&24573& 4000+& 72 &53&57&59\\
			\hline
		\end{tabular}
	\end{center}
	\caption{Number of iterations for the V-cycle method applied to the $\mathbb{Q}_3$ Lagrangian FEM Stiffness matrix, pre- and post-smoother 1 iteration of Gauss-Seidel with $\omega_{\rm pre,post}=1$, tolerance $\epsilon = 10^{-7}$.}
	\label{tab:MGM_Q3}
\end{table}

\begin{table}[htb]
	\begin{center}
		\begin{tabular}{cccccccccc}
			\hline
			$t$&$n=2^t-1$& N=$4 n$&$z=1$&$z=2$&$z=3$&$z=4$&$z=5$\\
			\hline
            3&	  7&   28&   81&  81 & 81& 81& 81\\
			4&	 15&   60&  177&  93 & 88& 90& 91\\
			5&	 31&   124&  395&  95 & 89& 91& 93\\
			6&	 63&  252&  988&  98 & 90& 93& 94\\
			7& 	127&  508& 2693& 103 & 92& 94& 96\\
			8& 	255&  1020& 4000+& 108 & 94& 96& 97\\
			9& 	511& 2044& 4000+& 114 & 95& 97& 99\\
			10& 1023& 4092& 4000+& 120 & 96& 99&100\\
			11& 2047& 8188& 4000+& 125 & 98&100&100\\
			12& 4095&16380& 4000+& 129 & 99&101&101\\
			13& 8191&32764& 4000+& 133 &101&101&101\\
			\hline
		\end{tabular}
	\end{center}
	\caption{Number of iterations for the V-cycle method applied to the $\mathbb{Q}_4$ Lagrangian FEM Stiffness matrix, pre- and post-smoother 1 iteration of Gauss-Seidel with $\omega_{\rm pre,post}=1$, tolerance $\epsilon = 10^{-7}$.}
	\label{tab:MGM_Q4}
\end{table}

\subsection{$\mathbb{Q}_{\deg}$ Lagrangian FEM stiffness matrices: the 2D case}\label{subsection:FEM2D}

Consider the uniform $\mathbb{Q}_{\deg}$ Lagrangian Finite Element approximation (FEM) of the second order elliptic differential problem
\begin{equation}\label{second_order_2D}
\begin{cases}
- \Delta u=\phi, & \text{ in } \Omega:=(0,1)^2,\\
\hfill u=0, & \text{ on }\partial\Omega,
\end{cases}
\end{equation}
where $\phi \in L^2(\Omega)$. Taking $n$ elements in each direction, the resulting stiffness matrix of size  $(\deg\cdot n-1)^2\times (\deg\cdot n-1)^2$ is 
$$\mathcal{A}_N = K_n^{(\deg)}\otimes M_n^{(\deg)}+M_n^{(\deg)}\otimes K_n^{(\deg)}, \qquad N = (\deg\cdot n-1)^2,$$
where $K_n^{(\deg)}$ and $M_n^{(\deg)}$ are the block-Toeplitz matrices
\begin{align}
K_n^{(\deg)}=T_n({f})_-,\qquad M_n^{(\deg)}=T_n({h})_-,\nonumber
\end{align}
with the subscript $-$ denoting again that the last row and column of $T_n({f})$ are removed. Explicit formulae for the matrix-valued trigonometric polynomials $f$ and $h$ and the spectral distribution of the matrices are given in~\cite{qp}.

In the following we want to apply the MGM strategy to the multilevel block-Toeplitz matrix ${A}_N$, for different choices of $\deg$. In the 1D case, we took the block-Toeplitz matrix with block size $\deg$. In the 2D case, we take the actual matrices arising from the considered FEM approximation of problem (\ref{second_order_2D}), which are not pure block-Toeplitz matrices with block size $\deg^2$. However, we can still apply our multigrid procedure due to its spectral properties given in in~\cite{qp}.

 Since the matrices are cut, also the projector slightly changes accordingly. In fact, we use the projectors $$[p_{\textbf{n}}^{\textbf{k}}]=[(T_{n}(p_z)(K_{n}^T\otimes I_{\deg})]_- \otimes [(T_{n}(p_z)(K_{n}^T\otimes I_{\deg})]_-,$$
 where $p_z$ is the univariate matrix-valued trigonometric polynomial of degree $c$  independent of $N$ defined in~(\ref{projector_pz}).
 
Extending the considerations that we made for the univariate case, we numerically look for the best choices of $z$ to obtain the optimality of the V-cycle method.
 
In Tables \ref{tab:MGM_Q2_2D}-\ref{tab:MGM_Q3_2D} we report for $z = 1,\dots,5$ the number of iterations needed for achieving the tolerance $\epsilon = 10^{-7}$ when increasing the matrix size and using $p_z$ in the construction of the projector. Table \ref{tab:MGM_Q2_2D} shows the results for the $\mathbb{Q}_2$ Lagrangian FEM Stiffness matrix and Table \ref{tab:MGM_Q3_2D} for the $\mathbb{Q}_3$ Lagrangian FEM Stiffness matrix. In both cases, we used as pre-smoother and post-smoother one iteration of Gauss-Seidel with $\omega_{\rm pre,post} = 1$. Moreover, we can see that the choice $z = 1$ does not yield optimality. For the other choices of $z$, conversely, the number of iterations needed for the MGM convergence remains almost constant, when increasing the size $N$. We numerically see that the best choice of $z$ is around 3 for both $\deg=2$ and $\deg= 3$.

\begin{table}[htb]
	\begin{center}
		\begin{tabular}{cccccccccc}
			\hline
			$t$&$n=2^t-1$& N=$(2n-1)^2$&$z=1$&$z=2$&$z=3$&$z=4$&$z=5$\\
			\hline
			3&	  7&    169&  62& 31 &22&20&19\\
			4&	 15&    841& 151& 40 &24&22&23\\
			5&	 31&   3721& 314& 42 &22&20&19\\
			6&	 63&  15625& 888& 51 &23&19&19\\
			7& 	127&  64009&2724& 63 &26&25&25\\
			8& 	255& 259081&4000+& 73 &27&23&22\\
			9& 	511&1042441&4000+& 80 &27&23&24\\
			10&1023&4182025&4000+& 84 &27&24&25\\
			\hline
		\end{tabular}
	\end{center}
	\caption{Number of iterations for the V-cycle method applied to the $\mathbb{Q}_2$ Lagrangian FEM Stiffness matrix for the two-dimensional problem, pre- and post-smoother 1 iteration of Gauss-Seidel with $\omega_{\rm pre,post}=1$, tolerance $\epsilon = 10^{-7}$.}
	\label{tab:MGM_Q2_2D}
\end{table}

\begin{table}[htb]
	\begin{center}
		\begin{tabular}{cccccccccc}
			\hline
			$t$&$n=2^t-1$& N=$(2n-1)^2$&$z=1$&$z=2$&$z=3$&$z=4$&$z=5$\\
			\hline
			3&	  7&    400& 143& 53 &53&53&54\\
			4&	 15&   1936& 326& 55 &53&54&54\\
			5&	 31&   8464& 886& 58 &52&53&53\\
			6&	 63&  35344&2719& 69 &57&59&60\\
			7& 	127& 144400&4000+& 83 &71&73&74\\
			8& 	255& 583696&4000+& 90 &60&60&60\\
			9& 	511&2347024&4000+& 94 &59&60&61\\
			\hline
		\end{tabular}
	\end{center}
	\caption{Number of iterations for the V-cycle method applied to the $\mathbb{Q}_3$ Lagrangian FEM Stiffness matrix for the two-dimensional problem, pre- and post-smoother 1 iteration of Gauss-Seidel with $\omega_{\rm pre,post}=1$, tolerance $\epsilon = 10^{-7}$.}
	\label{tab:MGM_Q3_2D}
\end{table}

\subsection{Matrices from staggered discontinuous Galerkin methods for the incompressible Navier-Stokes equations}\label{subsection:staggeredDG}
In this section we consider the matrices stemming from the discretization by staggered DG methods of the incompressible Navier-Stokes equations. This class of arbitrary high order accurate semi-implicit DG schemes on structured, adaptive Cartesian and unstructured edge-based staggered grids was proposed in \cite{CCY} and \cite{FD}.
In particular, we focus on the case where the degree of the polynomial Discontinuos Galerkin discretization ${\rm deg}$ is fixed and equal to $2$.

The incompressible Navier-Stokes equations consist in a divergence-free condition for the velocity
\begin{eqnarray*}
\nabla \cdot \vec{v}=0,
\end{eqnarray*}
and a momentum equation that involves non-linear convection, the pressure gradient and viscosity effects:
\begin{eqnarray*}
\diff{\vec{v}}{t}+\nabla \cdot \mathbf{F} + \nabla{\rm p}= \nabla \cdot \left( \nu \nabla \vec{v} \right).
\end{eqnarray*}
Here, $\vec{v}$ is the velocity field; ${\rm p}$ is the pressure; $\nu$ is the kinematic viscosity coefficient and $\mathbf{F}=\vec{v}\otimes \vec{v}$ is the tensor containing the non-linear convective term.

One of the crucial parts of the method proposed is to find the unknown pressure degrees of freedom at each time step. For a fixed dimension of the space equal to $2$ these unknowns can be obtained solving a large linear systems of form:
\begin{equation}\label{system}
A_Nx=b,  \quad x,b\in\mathbb{R}^{N},  \quad N=N(({\rm deg}+1)^2,\textbf{n}),
\end{equation}
where $\textbf{n}=(n_1,n_2)$ and $n_1,n_2$ are the total number of elements in each direction.
Consequently the coefficient matrix size $N$ grows to infinity as the approximation error tends to zero. 
In \cite{DFFMST} the structural properties of the positive definite matrix sequence $\{A_N\}$ has been studied. 

In particular, for ${\deg=2}$ and $\textbf{n}=(n,n)$, the $9n^2\times9n^2$ matrix $A_N$ can be decomposed as
\begin{equation}\label{K_N}
A_N=T_{\bf n}(f)+E_{\bf n}.
\end{equation}
Here $T_{\bf{n}}(f)$ is the Toeplitz matrix 
\begin{equation*}
T_{\bf{n}}(f) =\left[a_{\bx{i}-\bx{j}}\right]_{\bx{i},\bx{j}=\bx{e}}^{\bf n}
\end{equation*}
generated by $f:[-\pi,\pi]^2\rightarrow \C^{9\times9}$, with
\begin{equation}\label{eqn:symbol_SDG} 
f(\theta_1 , \theta_2 ) =  a_{(0,0)}+  a_{(-1,0)} e^{-\imath \theta_1}+  a_{(0,-1)} e^{-\imath \theta_2}+  a_{(1,0)} e^{\imath \theta_1}+  a_{(0,1)} e^{\imath \theta_2}.
\end{equation}
The matrix $E_{\bf n}$ is a low-rank perturbation, nonnegative definite and its rank grows at most proportionally to $n$. Hence, we focus on an optimal multigrid procedure for the linear system which has the Toeplitz matrix $T_{\textbf{n}}(f)$ as coefficient matrix.

Indeed, from \cite{ST}, if $A_n$ and $B_n$ are two positive definite matrices, with \[A_n\le\theta B_n,\] for some positive $\theta$ independent of $n$, then, if a multigrid procedure is optimal for the system with coefficent matrix $A_n$ then the same algorithm is optimal for the system with coefficent matrix $B_n$.

Moreover, in \cite{DFFMST} authors prove that $\lambda_{\rm min}(f(\theta_1,\theta_2))$ has a zero of order 2 at the origin and they exploit this information to propose a two grid procedure with a projector of the form
\begin{equation}
  p_{\textbf{n}}^{\textbf{k}}=(T_{n}(p_z)\otimes T_{n}(p_z)) \left(K^T_{\textbf{n}}\otimes I_9\right).
\end{equation} 
The latter is a natural extension in the multilevel block-Toeplitz setting of a projector of the form described in Subsection \ref{subsect:MGM} with $z=1$.
 
In Table \ref{tab:MGM_SDG}, we see that the same projector (with $z=1$) and smoother (Gauss-Seidel) do not yield an optimal MGM. However, the study of the ill-conditioning of the coarse problem suggests to try different values of $z$. Indeed, for $z = 2,\dots,5$ the number of iterations needed for achieving tolerance $\epsilon = 10^{-7}$ remains almost constant as the matrix size grows.

\begin{table}[htb]
	\begin{center}
		\begin{tabular}{cccccccccc}
			\hline
			$t$&$n=2^t-1$& N=$9 n^2$&$z=1$&$z=2$&$z=3$&$z=4$&$z=5$\\
			\hline
			3&	  7&   441&    13& 13 &13&13&13\\
			4&	 15&   2025&   19& 14 &14&15&15\\
			5&	 31&   8649&   36& 15 &15&16&17\\
			6&	 63&  35721&   83& 17 &16&18&20\\
			7& 	127&  145161& 220& 18 &17&20&21\\
			8& 	255&  585225& 635& 19 &19&22&23\\
			\hline
		\end{tabular}
	\end{center}
	\caption{Number of iterations for the V-cycle method applied to the staggered DG matrix for the incompressible Navier-Stokes equations, with pre- and post-smoother 1 iteration of Gauss-Seidel with $\omega_{\rm pre,post}=1$, tolerance $\epsilon = 10^{-7}$.}
	\label{tab:MGM_SDG}
\end{table}

\section{Conclusions and Future Developements}\label{sect:final}

In the past decades, multigrid methods for linear systems having multilevel Toeplitz coefficient matrices with scalar entries have been largely studied. Conversely, the case of block entries has been considered only for specific applications and without taking care of a general convergence theory. Here the main aim was to start filling this gap. 
The theoretical analysis indicates that the generalization is not trivial since the commutativity played an essential role in the scalar case and here is cannot be used.

Among the numerous applications that lead to the block Toeplitz structure, we have considered high order Lagrangian FEM and staggered DG methods. The numerical results have confirmed the effectiveness of our proposal and the consistency of the  proposed theoretical analysis.

We observe that our theoretical results can be useful to mathematically support the projection  strategies proposed in several applications. For example the choice of the projector for  tensor rectangular FEM $\mathbb{Q}_{\deg}$ approximations of any dimension $\Bbbk$ based on a geometric approach \cite{FRTBS}.

Among the open problems we can list the full convergence analysis for the V-cycle, a deeper analysis of role of the non-commutativity in the block setting, and the choice of more efficient smoothers especially in the multilevel setting. In fact, in the case of multivariate PDE, we encounter multilevel block structures and the computational cost of Gauss-Seidel is too high for the method to be competitive with existing solvers, since the bandwidth of the matrix depends on the matrix-size. We remind that we used it in our numerical computations, just for showing the robustness of the projectors, but an efficient choice of the smoothers is computationally important and it has to be the subject of future investigations.

\section{Acknowledgements}
This work was supported by Gruppo Nazionale per il Calcolo Scientifico (GNCS-INdAM).

\end{document}